\documentclass{article}
\usepackage{amsmath}
\usepackage{amsfonts}
\usepackage{amsthm}
\usepackage{amssymb}
\usepackage{color}

\newtheorem{theorem}{Theorem}[section]
\newtheorem{lemma}[theorem]{Lemma}

\newtheorem{observation}[theorem]{Observation}

\theoremstyle{definition}
\newtheorem{definition}[theorem]{Definition}
\theoremstyle{remark}
\newtheorem{remark}[theorem]{Remark}

\newcommand\remove[1]{}

\def\f2{\mathbb{F}_2}

\def\lip{\hskip0.02cm{\rm Lip}\hskip0.01cm}

\newcommand{\conv}{{\rm conv}\hskip0.02cm}

\newcommand{\ep}{\varepsilon}

\newcommand{\diam}{{\rm diam}\hskip0.02cm}

\begin{document}
\title{\LARGE{Test-space
characterizations of some classes of Banach spaces}}

\author{M.\,I.~Ostrovskii}

\date{\today}
\maketitle

\noindent{\bf Abstract.} Let $\mathcal{P}$ be a class of Banach
spaces and let $T=\{T_\alpha\}_{\alpha\in A}$ be a set of metric
spaces. We say that $T$ is a set of {\it test-spaces} for
$\mathcal{P}$ if the following two conditions are equivalent: {\bf
(1)} $X\notin\mathcal{P}$; {\bf (2)} The spaces
$\{T_\alpha\}_{\alpha\in A}$ admit uniformly bilipschitz
embeddings into $X$.
\medskip

The first part of the paper is devoted to a simplification of the
proof of the following test-space characterization obtained in
M.\,I.~Ostrovskii [Different forms of metric characterizations of
classes of Banach spaces, {\it Houston J. Math.}, to appear]:
\medskip

For each sequence $\{X_m\}_{m=1}^\infty$ of finite-dimensional
Banach spaces there is a sequence $\{H_n\}_{n=1}^\infty$ of finite
connected unweighted graphs with maximum degree $3$ such that the
following conditions on a Banach space $Y$ are equivalent: {\bf
(A)}  $Y$ admits uniformly isomorphic embeddings of
$\{X_m\}_{m=1}^\infty$; {\bf (B)} $Y$ admits uniformly bilipschitz
embeddings of $\{H_n\}_{n=1}^\infty$.
\medskip

The second part of the paper is devoted to the case when
$\{X_m\}_{m=1}^\infty$ is an increasing sequence of spaces. It is
shown that in this case the class of spaces given by {\bf (A)} can
be characterized using one test-space, which can be chosen to be
an infinite graph with maximum degree $3$.
\medskip

\noindent{\bf 2010 Mathematics Subject Classification:} Primary:
46B07; Secondary: 05C12, 46B85, 54E35

\begin{large}

\tableofcontents

\section{Introduction}

Embeddings of metric spaces into Banach spaces play an important
role in Computer Science (see, for example, \cite[Chapter
15]{WS11}) and in Topology (see \cite{Yu06}). In connection with
problems of embeddability of metric spaces into Banach spaces it
would be interesting to find metric characterizations of
well-known classes of Banach spaces. By {\it metric
characterizations} we mean characterizations which can be tested
on an arbitrary metric space. So, in a metric characterization
only distances between elements of the space are involved, and no
linear combinations of any kind can be used. At this point it
seems to be unclear: how to define the notion of a metric
characterization which would be the most useful for applications
in the theory of metric embeddings? One can try to define a metric
characterization in the following way: a {\it metric
characterization} is a set of {\it formulas} understood as in
logic (see \cite[p.~19]{Han77} for a definition of a first-order
formula). Some of the variables in the formulas are elements in an
unknown metric space $X$ (so the formulas make sense for an
arbitrary metric space $X$). We say that such set of formulas {\it
characterizes} a class $\mathcal{P}$ of Banach spaces if the
following conditions are equivalent for a Banach space $X$:

\begin{itemize}

\item $X\in\mathcal{P}$

\item All of the formulas of the set hold for $X$.

\end{itemize}

Metric characterizations which we are going to study in this paper
use the following definition.

\begin{definition} {\rm
Let $0\le C<\infty$. A map $f: (A,d_A)\to (Y,d_Y)$ between two
metric spaces is called $C$-{\it Lipschitz}\label{lip} if
\[\forall u,v\in A\quad d_Y(f(u),f(v))\le Cd_A(u,v).\] A map $f$
is called {\it Lipschitz} if it is $C$-Lipschitz for some $0\le
C<\infty$. For a Lipschitz map $f$ we define its {\it Lipschitz
constant}\label{lipC} by
\[\lip f:=\sup_{d_A(u,v)\ne 0}\frac{d_Y(f(u),
f(v))}{d_A(u,v)}.\label{lipCE}\]

Let $1\le C<\infty$. A map $f:A\to Y$ is called a {\it
$C$-bilipschitz embedding} if there exists $r>0$ such that
\begin{equation}\label{E:MapDist}\forall u,v\in A\quad rd_A(u,v)\le
d_Y(f(u),f(v))\le rCd_A(u,v).\end{equation} A {\it bilipschitz
embedding} is an embedding which is $C$-bilipschitz for some $1\le
C<\infty$. The smallest constant $C$ for which there exist $r>0$
such that \eqref{E:MapDist} is satisfied is called the {\it
distortion} of $f$. A set of embeddings is called {\it uniformly
bilipschitz} \,if they have uniformly bounded distortions.}
\end{definition}

\begin{remark} Linear embeddings $T_i:A_i\to Y$ of Banach (or normed) spaces
into a Banach (normed) space are uniformly bilipschitz if and only
if
\[\sup_i\left(||T_i||\cdot||T_i^{-1}|_{T_i(A_i)}||\right)<\infty.\]
Such embeddings $T_i$ are called {\it uniformly isomorphic}.
\end{remark}

\begin{remark}\label{R:TrivNonRef}
The definition of a metric characterization suggested above does
not seem to be completely satisfactory. It includes trivial (in a
certain sense) characterizations of the type: A Banach space is
nonreflexive if and only if it contains a (separable) subset which
is bilipschitz equivalent to a nonreflexive separable Banach
space. (The validity of this characterization is a consequence of
the following well-known facts: (1) Each nonreflexive Banach space
contains a separable nonreflexive subspace. This fact follows, for
example, from the Eberlein--Shmulian theorem \cite[Theorem
V.6.1]{DS58}; (2) If a Banach space $Y$ admits a bilipschitz
embedding of a (separable) nonreflexive Banach space $X$, then $Y$
is nonreflexive, see \cite[Lemma 3.1]{HM82} or \cite[Theorem
7.9]{BL00}.)
\end{remark}

At this point it is not clear how to define a metric
characterization in such a way that, on one hand, all interesting
examples are included, but, on the other hand, characterizations
like the trivial characterization of nonreflexivity mentioned in
Remark \ref{R:TrivNonRef} are excluded. We shall focus on one of
the classes of metric characterizations which is known to be
interesting (see
\cite{Bau07,Bau09+,Bou86,BMW86,JS09,MN08,Ost11a,Ost11d,Pis86}). We
mean metric characterizations of the following type.

\begin{definition}\label{D:TestSp} Let $\mathcal{P}$ be a class of Banach spaces
and let $T=\{T_\alpha\}_{\alpha\in A}$ be a set of metric spaces.
We say that $T$ is a set of {\it test-spaces} for $\mathcal{P}$ if
the following two conditions are equivalent
\begin{enumerate}
\item\label{I:NotInP} $X\notin\mathcal{P}$.

\item\label{I:DefTesSp2} The spaces $\{T_\alpha\}_{\alpha\in A}$
admit uniformly bilipschitz embeddings into $X$.
\end{enumerate}
\end{definition}

\begin{remark} We use $X\notin\mathcal{P}$ in condition \ref{I:NotInP} of Definition \ref{D:TestSp} rather than
$X\in\mathcal{P}$ for terminological reasons: we would like to use
terms ``test-spaces for reflexivity, superreflexivity, etc.''
rather than ``test-spaces for {\bf non}reflexivity, {\bf
non}superreflexivity, etc.''\end{remark}

\begin{remark} One can introduce the notion of test-spaces differently,
requiring, for example: ``at least one of the spaces
$\{T_\alpha\}_{\alpha\in A}$ admits a bilipschitz embedding into
$X$''. However, this version of test-space characterizations
includes the trivial characterization of reflexivity mentioned in
Remark \ref{R:TrivNonRef}. Another reason why we have chosen the
introduced in Definition \ref{D:TestSp} notion of test-spaces is:
many important known characterizations are of this form (see
\cite{Bau07,Bau09+,Bou86,BMW86,JS09,MN08,Ost11a,Ost11d,Pis86}).
\end{remark}

\section{Simplification of the proof on metric characterizations
of classes with excluded subspaces}\label{S:Houston}

The main purpose of this section is to give a simpler proof of the
following result of \cite{Ost11d}:

\begin{theorem}[\cite{Ost11d}]\label{T:Ost11+} For each sequence $\{X_m\}_{m=1}^\infty$
of finite-dimensional Banach spaces there exists a sequence
$\{H_n\}_{n=1}^\infty$ of finite connected unweighted graphs with
maximum degree $3$ such that the following conditions on a Banach
space $Y$ are equivalent:
\begin{itemize}
\item $Y$ admits uniformly isomorphic embeddings of
$\{X_m\}_{m=1}^\infty$. \item $Y$ admits uniformly bilipschitz
embeddings of $\{H_n\}_{n=1}^\infty$.
\end{itemize}
\end{theorem}

Everywhere in this paper we consider graphs as metric spaces with
their {\it shortest path metric}: the distance between two
vertices is the length of the shortest path between them. In some
cases we consider weighted graphs with some positive weights
assigned to their edges. In such a case the length of the path is
the sum of weights of edges included in it. For graphs with no
weights (sometimes we emphasize this by calling them {\it
unweighted graphs}) the length of a path is the number of edges in
it (this corresponds to the case when all weights are equal to
$1$).

\begin{remark}\label{R:InfDim} For the reasons explained in \cite{Ost11d} we restrict our
attention to the case $\sup_m\dim X_m=\infty$.
\end{remark}

Our purpose is to simplify the proof of the step which has the
longest proof in \cite{Ost11d}. We shall also present other steps
of the proof, in a more general form than in \cite{Ost11d}. The
reason for doing so is that we need these steps in the more
general form later in the paper. Recall some standard definitions.

\begin{definition}\label{D:MetrNotions} Let $\alpha>0$. We say that a subset $A$ in a metric space $(X,d)$ is
{\it $\alpha$-dense} in a subset $B\subset X$ if $A\subset B$ and
\[\forall x\in B~\exists y\in A~d(x,y)\le\alpha.\]

A subset $D$ in a metric space $(X,d)$ is  called {\it
$\alpha$-separated} if $d(x,y)\ge\alpha$ for each $x,y\in D$,
$x\ne y$.
\medskip

If $A$ and $B$ are subsets in a metric space $(X,d)$, we let
\[d(A,B)=\inf\{d(x,y):~x\in A,~ y\in B\}.\]
\end{definition}
\medskip

First we introduce an approximate description of convex sets in
Banach spaces using unweighted graphs. Let $C$ be a convex set in
a Banach space $X$, $\alpha,\beta>0$, and let $V$ be an
$\alpha$-separated $\beta$-dense subset of $C$.

\begin{remark} It is easy to see that such subset $V$ does not exist if
$\beta<\frac{\alpha}2$ and $\diam C>\beta$. In this paper only the
case where $\beta\ge\alpha$ is considered. In this case the
existence is immediate from Zorn's lemma.\end{remark}

\begin{definition}\label{D:CAlphaBeta} {\rm
Let $G$ be the graph whose vertex set is $V$ and whose edge set is
defined in the following way: vertices $u, v\in V$ are joined by
an edge if and only if $||u-v||\le 3\beta$. The graph $G$ is
called a {\it $(C,\alpha,\beta)$-graph}. If $\alpha=\beta$, $G$ is
called a {\it $(C,\alpha)$-graph}.}
\end{definition}

It is easy to check that $G$ is not uniquely determined by $C$,
$\alpha$, and $\beta$, but for our purposes it does not matter
which of $(C,\alpha,\beta)$-graphs we pick. We endow the vertex
set $V$ of $G$ with its shortest path metric $d_G$.

\begin{lemma}\label{L:CAlphaBeta} The natural embedding $f:(V,d_G)\to (X,||\cdot||)$ is bilipschitz
with distortion $\le\max\left\{3,\frac{3\beta}\alpha\right\}$.
More precisely,
\[\lip(f)\le 3\beta~\hbox{and}~
\lip(f^{-1}|_{f(V)})\le\max\left\{\frac1\beta,\frac1\alpha\right\}.\]
\end{lemma}

\begin{proof} The inequality $\lip(f)\le 3\beta$ follows
immediately from the fact that adjacent vertices in $G$ are at
distance $\le3\beta$ in $X$, and the definition of the shortest
path metric.
\medskip

To prove the inequality for $\lip(f^{-1})$ we consider two
distinct vertices $u,v\in V$, write $||u-v||=d\beta$ for some
$d>0$, and consider two cases:\medskip

\noindent{\bf Case A.} $d\le 3$. In such a case $d_G(u,v)=1$.
Since $||u-v||\ge\alpha$, we have
\[\frac{d_G(u,v)}{||u-v||}\le\frac1\alpha.\]

\noindent{\bf Case B.} $d>3$. In this case, and even in a wider
case $d>2$, we show that
\begin{equation}\label{E:CaseB}d_G(u,v)\le\lfloor
d\rfloor-1\end{equation} and therefore
\[
\frac{d_G(u,v)}{||u-v||}\le\frac{\lfloor
d\rfloor-1}{d\beta}\le\frac1\beta.
\]

We prove the inequality \eqref{E:CaseB} by induction starting with
$2<d\le 3$. In this case $d_G(u,v)=1$ and so it is clear that
\eqref{E:CaseB} is satisfied.

Suppose that we have proved the inequality \eqref{E:CaseB} for
$2<d\le n$. Let us show that this implies the inequality for
$n<d\le n+1$. We do this as follows:

Consider the vertex $\widetilde u$ lying on the line segment
joining $u$ and $v$ at distance $2\beta$ from $u$. Since
$\widetilde u\in C$ (this is the point where we use the convexity
of $C$), there is $w\in V$ satisfying $||w-\widetilde
u||\le\beta$.\medskip

By the triangle inequality, we have $||w-u||\le 3\beta$ and
$||w-v||\le (d-1)\beta$. The first inequality implies
$d_G(w,u)=1$. Applying the triangle inequality and the Induction
Hypothesis, we get $d_G(u,v)\le d_G(w,v)+1\le(\lfloor
d-1\rfloor-1)+1=\lfloor d\rfloor-1$. (We cannot apply the
Induction Hypothesis if $||w-v||\le 2\beta$, but in this case the
result is also easy to verify.)
\end{proof}

If $X$ is a Banach space, we use the notation $B_X(r)$, $r>0$, for
$\{x\in X:~||x||\le r\}$. The unit ball $B_X(1)$ is also denoted
by $B_X$. Observe that if $X$ is finite-dimensional (and
$\beta\ge\alpha$), $(B_X(r),\alpha,\beta)$-graphs are finite.
\medskip

The next step in the proof of Theorem \ref{T:Ost11+} in
\cite{Ost11d} is the following lemma (in \cite{Ost11d} the lemma
is stated in slightly less general form).

\begin{lemma}\label{L:GraphsToSubspaces} If $\{X_m\}_{m=1}^\infty$ are finite-dimensional Banach spaces,
and a Banach spa\-ce $Y$ admits uniformly bilipschitz embeddings
of a collection of
$\left(B_{X_m}(n),\alpha(n),\beta(n)\right)$-graphs
$m,n\in\mathbb{N}$, where $\alpha(n)\le\beta(n)$,
$\lim_{n\to\infty} \alpha(n)=0$, $\lim_{n\to\infty}\beta(n)=0$,
and $\sup_n\frac{\beta(n)}{\alpha(n)}<\infty$, then
$\{X_m\}_{m=1}^\infty$ are uniformly isomorphic to subspaces of
$Y$.
\end{lemma}

Lemma \ref{L:GraphsToSubspaces} can be derived from the following
discretization result.

\begin{theorem}\label{T:discr} For each finite-dimensional Banach space
$X$ and each $\gamma>0$ there exists $\ep>0$ such that for each
bilipschitz embedding $L$ of an $\ep$-dense subset of $B_X$, with
the metric inherited from $X$, into a Banach space $Y$ there is a
linear embedding $T:X\to Y$ such that
$(1-\gamma)||T||\cdot||T^{-1}||\le($the distortion of $L)$.
\end{theorem}

This theorem goes back to Ribe \cite{Rib76}, a new proof of the
essential ingredients was found by Heinrich-Mankiewicz
\cite{HM82}, versions of these proofs are presented in
\cite{BL00}. The first explicit bound on $\ep$ was found by
Bourgain \cite{Bou87}. Bourgain's proof was simplified and
explained by Begun \cite{Beg99} and Giladi-Naor-Schechtman
\cite{GNS11}. We recommend everyone who would like to study this
result to start by reading \cite{GNS11}.\medskip

It is clear that the fact that the ball in Theorem \ref{T:discr}
has radius $1$ plays no role. To derive Lemma
\ref{L:GraphsToSubspaces} from Theorem \ref{T:discr} we observe
that the vertex set of a
$\left(B_{X_m}(n),\alpha(n),\beta(n)\right)$-graph is a
$\beta(n)$-dense subset of $B_{X_m}(n)$. Furthermore, by Lemma
\ref{L:CAlphaBeta}, the graph distance on this set is
$\max\left\{3,\frac{3\beta(n)}{\alpha(n)}\right\}$-equivalent to
the metric inherited from $X$. This proves Lemma
\ref{L:GraphsToSubspaces}.\medskip

Lemmas \ref{L:CAlphaBeta} and \ref{L:GraphsToSubspaces} show that
a Banach space $Y$ admits uniformly isomorphic embeddings of $X_m$
if and only if $Y$ admits uniformly bilipschitz embeddings of some
(or any) collection of
$\left(B_{X_m}(n),\alpha(n),\beta(n)\right)$-graphs
$m,n\in\mathbb{N}$, where $\alpha(n)\le\beta(n)$,
$\lim_{n\to\infty} \alpha(n)=0$, $\lim_{n\to\infty}\beta(n)=0$,
and $\sup_n\frac{\beta(n)}{\alpha(n)}<\infty$. This statement does
not complete the proof of Theorem \ref{T:Ost11+} (even of a weaker
version of it, with $3$ replaced by any other uniform bound on
degrees of $\{H_n\}_{n=1}^\infty$). In fact, it is easy to see
that if $\sup_m\dim X_m=\infty$, the degrees of any collection of
$\left(B_{X_m}(n),\alpha(n),\beta(n)\right)$-graphs,
$m,n\in\mathbb{N}$, where $\alpha(n)\le\beta(n)$,
$\lim_{n\to\infty} \alpha(n)=0$, $\lim_{n\to\infty}\beta(n)=0$,
and $\sup_n\frac{\beta(n)}{\alpha(n)}<\infty$, are unbounded.

\begin{observation}\label{O:IntoInfDim}
Since we assumed $\sup_m\dim X_m=\infty$, each Banach space
admitting uniformly isomorphic embeddings of $X_m$ has to be
infinite-dimensional. For this reason, in order to prove Theorem
\ref{T:Ost11+} it suffices to show that for each
finite-dimensional Banach space $X$, each $(B_X(r),\delta)$-graph
$G$ $(0<\delta<r<\infty)$, and each infinite-dimensional Banach
space $Z$ containing $X$ as a subspace, there exist a graph $H$
and bilipschitz embeddings $\psi:G\to H$ and $\varphi:H\to Z$,
such that distortions of $\psi$ and $\varphi$ are bounded from
above by absolute constants and the maximum degree of $H$ is $3$.
It is clear that it is enough to consider the case $\delta=1$.
\end{observation}

\begin{remark} Observation \ref{O:IntoInfDim} is the main step
towards simplification of the proof of Theorem \ref{T:Ost11+}
given in \cite{Ost11d}: in \cite{Ost11d} the graph $H$ was
embedded into $X$ (if $\dim X\ge 3$). This is substantially more
difficult and, as we see from Observation \ref{O:IntoInfDim}, is
not needed to prove Theorem \ref{T:Ost11+}. However, the proof of
\cite[Section 4]{Ost11d} could be of independent interest, see
Remark \ref{R:Thickening} below.
\end{remark}

The construction of $H$ which we use is the same as in
\cite{Ost11d}: We introduce the graph $MG$ as the following
``expansion'' of $G$: we replace each edge in $G$ by a path of
length $M$. It is clear that the graph $MG$ is well-defined for
each $M\in\mathbb{N}$. In our construction of $H$ the number $M$
will be chosen to be much larger than the number of edges of $G$.
We use the term {\it long paths} for the paths of length $M$ which
replace edges of $G$. Next step in the construction of $H$: For
each vertex $v$ of $G$, we introduce a path $p_v$ in the graph $H$
whose length is equal to the number of edges of $G$, we call each
such path a {\it short path}. At the moment these paths do not
interact. We continue our construction of $H$ in the following
way. We label vertices of short paths in a monotone way by long
paths. ``In a monotone way'' means that the first vertex of {\bf
each} short path corresponds to the long path $p_1$, the second
vertex of {\bf each} short path corresponds to the long path $p_2$
etc. We complete our construction of $H$ introducing, for a long
path $p$ in $MG$ corresponding to an edge $uv$ in $G$, a path of
the same length in $H$ (we also call it {\it long}) which joins
those vertices of the short paths $p_u$ and $p_v$ which have label
$p$. There is no further interaction between short and long paths
in $H$. It is obvious that the maximum degree of $H$ is $3$.
\medskip

It remains to define embeddings $\psi$ and $\varphi$ and to
estimate their distortions.\medskip

To define $\psi$ we pick a long path $p$ in $MG$ (in an arbitrary
way) and map each vertex $u$ of $G$ onto the vertex in $H$ having
label $p$ in the short path $p_u$ corresponding to $u$. The
estimates for $\lip(\psi)$ and $\lip(\psi^{-1})$ given below are
taken from \cite{Ost11d}. We reproduce them because they do not
take much space.\medskip

We have $\lip(\psi)\le 2e(G)+M$, where $e(G)$ is the number of
edges of $G$. In fact, to estimate the Lipschitz constant it
suffices to find an estimate from above for the distances in $H$
between $\psi(u)$ and $\psi(v)$ where $u$ and $v$ are adjacent
vertices of $G$. To see that $2e(G)+M$ provides the desired
estimate we consider the following three-stage walk from $\psi(u)$
from $\psi(v)$:
\begin{itemize}
\item We walk from $\psi(u)$ along the short path $p_u$ to the
vertex labelled by the long path corresponding to the edge $uv$ in
$G$. \item Then we walk along the corresponding long path to its
end in $p_v$. \item We conclude the walk with the piece of the
short path $p_v$ which we need to traverse in order to reach
$\psi(v)$.
\end{itemize}
\smallskip

We claim that $\lip(\psi^{-1})\le M^{-1}$. This gives an absolute
upper bound for the distortion of $\psi$ provided the quantity
$e(G)$ is controlled by $M$, we need $M$ to be much larger than
$e(G)$ only if we would like to make the distortion close to $1$.
To prove $\lip(\psi^{-1})\le M^{-1}$ we let $\psi(u)$ and
$\psi(v)$ be two vertices of $\psi(V(G))$. We need to estimate
$d_G(u,v)$ from below in terms of $d_H(\psi(u),\psi(v))$. Let
\[P=\psi(u),w_1,\dots,w_{n-1},\psi(v)\] be one of the shortest
$\psi(u)\psi(v)$-paths in $H$. Let $u,u_1,\dots, u_{k-1},v$ be
those vertices of $G$ for which the path $P$ visits the
corresponding short paths $p_u,p_{u_1},\dots,p_{u_{k-1}},p_v$. We
list $u_1,\dots,u_{k-1}$ in the order of visits. It is clear that
in such a case the sequence $u,u_1,\dots,u_{k-1},v$ is a $uv$-walk
in $G$. Therefore $d_G(u,v)\le k$. On the other hand, in $H$, to
move from one short path to another, one has to traverse at least
$M$ edges, therefore $d_H(\psi(u),\psi(v))\ge kM$. This implies
$\lip(\psi^{-1})\le M^{-1}$.\medskip

Our next purpose is to introduce $\varphi:H\to Z$. First we prove
(Lemma \ref{L:TGintoOver} below) that there is a bilipschitz
embedding of $MG$ into some finite-dimensional subspace $W$ of $Z$
with distortion bounded by an absolute constant.\medskip

It is convenient to handle all $M\in \mathbb{N}$ simultaneously by
considering the following thickening of the graph $G$ (see
\cite[Section 1.B]{Gro93} for the general notion of thickening).
For each edge $uv$ in $G$ we join $u$ and $v$ with a set isometric
to $[0,1]$, we denote this set $t(uv)$. The {\it thickening} $TG$
of $G$ is the union of all sets $t(uv)$ (such sets can intersect
at their ends only) with the distance between two points defined
as the length of the shortest curve joining the points.\medskip

\begin{lemma}\label{L:TGintoOver} If a finite unweighted  graph $G$ endowed with
its graph distance admits a bilipschitz embedding $\tau$ into a
finite-dimensional Banach space $X$, then the graph $TG$ admits a
bilipschitz embedding $f$ into any infinite-dimensional Banach
space $Z$ containing $X$ as a subspace, and the distortion of $f$
is bounded in terms of the distortion of $\tau$ and some absolute
constants.
\end{lemma}

\begin{remark}\label{R:Thickening} In \cite[Section 4]{Ost11d} a
stronger result was proved, namely, it was proved that in the case
where $\dim X\ge 3$, the bilipschitz embedding $f$ whose existence
is claimed in Lemma \ref{L:TGintoOver} can be required to map $TG$
into $X$. As we shall see, this result is not needed for Theorem
\ref{T:Ost11+}. However, it could be of independent interest.
\end{remark}

\begin{proof}[Proof of Lemma \ref{L:TGintoOver}] We may assume without loss of generality that
$\lip(\tau^{-1})=1$, that is, $||\tau(u)-\tau(v)||\ge d_G(u,v)$.
We construct a bilipschitz embedding of $TG$ into an (arbitrary)
Banach space $W$ containing $X$ as a subspace and satisfying $\dim
(W/X)=e(G)$, where $e(G)$ is the number of edges in $G$ and $W/X$
is the quotient space.\medskip

We find an Auerbach basis in $W/X$. Recall the definition. Let
$\{x_i\}_{i=1}^n$ be a basis in an $n$-dimensional Banach space
$Y$, its {\it biorthogonal functionals} are defined by
$x_i^*(x_j)=\delta_{ij}$ (Kronecker delta). The basis
$\{x_i\}_{i=1}^n$ is called an {\it Auerbach basis} if
$||x_i||=||x^*_i||=1$ for all $i\in\{1,\dots,n\}$. This notion
goes back to \cite{Aue30}. See \cite[Section 2]{Ost11b} and
\cite{Pli95} for historical comments and proofs.\medskip

Since the cardinality of the Auerbach basis is equal to the number
of edges in $G$, we label its elements by edges. Also we lift the
elements of this Auerbach basis into $W$. Since $W$ is
finite-dimensional, we may assume that the norms of the lifted
elements are also equal to $1$. We use the notation
$\{e_{uv}\}_{uv\in E(G)}$ for the lifted elements of the Auerbach
basis and the notation $\{e_{uv}^*\}_{uv\in E(G)}$ for its
biorthogonal system. It is clear that $\{e_{uv}^*\}_{uv\in E(G)}$
may be regarded as elements of $W^*$.\medskip

If $uv$ is an edge in $G$, we map $t(uv)$ onto the concatenation
of two line segments in $W\subset Z$, namely, onto the
concatenation of
$\left[\tau(u),\frac{\tau(u)+\tau(v)}2+e_{uv}\right]$ and
$\left[\frac{\tau(u)+\tau(v)}2+e_{uv},\tau(v)\right]$. More
precisely, we map the point in $t(uv)$ which is at distance
$\alpha\in[0,1]$ from $u$ onto the point
\begin{equation}\label{E:f}\alpha\tau(v)+(1-\alpha)\tau(u)+\min\{2\alpha,2(1-\alpha)\}e_{uv}.\end{equation}
It is easy to check that such points cover the concatenation of
the line segments
$\left[\tau(u),\frac{\tau(u)+\tau(v)}2+e_{uv}\right]$ and
$\left[\frac{\tau(u)+\tau(v)}2+e_{uv},\tau(v)\right]$. We denote
this map by $f:TG\to W$.\medskip

We claim that $f$ is a bilipschitz embedding and its distortion is
bounded in terms of distortion of $\tau$ and some absolute
constant. To estimate $\lip(f)$ observe that the derivative of the
function in \eqref{E:f} with respect to $\alpha$ is
$\tau(v)-\tau(u)\pm 2e_{uv}$ (at points where the derivative is
defined). Hence $\lip(f)\le\lip(\tau)+2$.
\medskip

To estimate the Lipschitz constant of $f^{-1}$, we need to
estimate from above the quotient
\[\frac{d_{TG}(x,y)}{||f(x)-f(y)||},\]
where $x,y\in TG$. Let $\alpha$ be the distance in $TG$ from $x\in
t(uv)$ to $u$ and let $\beta$ be the distance from $y\in t(wz)$ to
$w$. We may choose our notation in such a way that
$\alpha,\beta\le\frac12$. Let $D=d_{TG}(x,y)$.

First we consider the case when the edges $uv$ and $wz$ are
different. We have $d_{TG}(u,w)\ge D-\alpha-\beta$. Thus
$||f(u)-f(w)||=||\tau(u)-\tau(w)||\ge D-\alpha-\beta$ and
$||f(x)-f(y)||\ge D-(\lip(\tau)+3)(\alpha+\beta)$. Here we use the
fact that, since $\lip(f)\le\lip(\tau)+2$, we have
$||f(x)-f(u)||\le(\lip(\tau)+2)\alpha$ and $||f(y)-f(w)||\le
(\lip(\tau)+2)\beta$. On the other hand $||f(x)-f(y)||\ge
e^*_{uv}(f(x)-f(y))=e^*_{uv}(f(x)-f(u))+e^*_{uv}(f(u)-f(y))=e^*_{uv}(f(x)-f(u))=
2\alpha$ (we use $uv\ne wz$, the definition of $\alpha$ and the
fact that $\{e_{uv}\}$ is a lifted Auerbach basis in $W/X$).
Similarly we get $||f(x)-f(y)||\ge2\beta$. Therefore

\[\frac{d_{TG}(x,y)}{||f(x)-f(y)||}\le\min\left
\{\frac{D}{\max\{0,D-(\lip(\tau)+3)(\alpha+\beta)\}},
\frac{D}{2\alpha},\frac{D}{2\beta}\right\}.\] It is easy to see
that the minimum in this inequality is bounded from above in terms
of $\lip(\tau)$ and an absolute constant. If $\alpha=0$, or
$\beta=0$, or both, we modify this argument in a straightforward
way.\medskip

It remains to consider the case where $x,y\in t(uv)$. Let
$d_{TG}(x,u)=\alpha$ and $d_{TG}(y,u)=\beta$, so
$d_{TG}(x,y)=|\alpha-\beta|$. It is easy to see that
\[|e^*_{uv}(f(x)-f(y))|=\begin{cases} 2\,|\alpha-\beta| & \hbox{ if
$\alpha$ and $\beta$ are on the same side of $\frac12$}\\
2\,|1-\alpha-\beta| & \hbox{ otherwise.} \end{cases}\] In the
former case  we get
\[\frac{d_{TG}(x,y)}{||f(x)-f(y)||}\le\frac 12.\]
If the latter case we use
\[f(x)-f(y)=(\alpha-\beta)(\tau(v)-\tau(u))\pm
2(1-\alpha-\beta)e_{uv}.\] We get from here that
\[||f(x)-f(y)||\ge \max\{2|1-\alpha-\beta|,
|\alpha-\beta|||\tau(u)-\tau(v)||-2|1-\alpha-\beta|\}.\] The
desired estimate follows.
\end{proof}

This proof shows that there exists an embedding $\varphi_0:MG\to
W$ such that $\lip(\varphi_0)\le 1$ and $\lip(\varphi_0^{-1})$ is
bounded from above by an absolute constant. The rest of the proof
is quite similar to the proof in \cite{Ost11d}, we only replace
the embedding $\varphi_0:MG\to X$ constructed in \cite{Ost11d} by
the embedding $\varphi_0:MG\to W$ constructed in Lemma
\ref{L:TGintoOver}. For convenience of the reader we reproduce
this part of the proof with necessary modifications. We number
vertices along short paths using numbers from $1$ to $e(G)$ in
such a way that vertices numbered $1$ correspond to the same long
path in the correspondence described above.
\medskip

At this point we are ready to describe the action of the map
$\varphi$ on vertices of short paths. We construct the map
$\varphi$ as a map into the Banach space $W\oplus_1\mathbb{R}$.
This is enough because $W\oplus_1\mathbb{R}$ admits a linear
bilipschitz embedding into $Z$ with distortion bounded by an
absolute constant.
\medskip

For vertex $w$ of $H$ having number $i$ on the short path $p_u$
the image of $w$ in $W\oplus_1\mathbb{R}$ is
$\varphi(w)=\varphi_0(u)\oplus i$ (here we use the same notation
$u$ both for a vertex of $G$ and the corresponding vertex in
$MG$).
\medskip

To map vertices of long paths of $H$ into $W\oplus_1\mathbb{R}$ we
observe that the numbering of vertices of short paths leads to a
one-to-one correspondence between long paths and numbers
$\{1,\dots,e(G)\}$. We define the map $\varphi$ on a long path
corresponding to $i$ by $\varphi(w)=\varphi_0(w')\oplus i$, where
$w'$ is the uniquely determined vertex in a long path of $MG$
corresponding to a vertex $w$ in a long path of $H$.
\medskip

The fact that $\lip(\varphi)\le 1$ follows immediately from the
easily verified claim that the distance between $\varphi$-images
of adjacent vertices of $H$ is at most $1$ (here we use
$\lip(\varphi_0)\le 1$).\medskip

We turn to an estimate of $\lip(\varphi^{-1})$. In this part of
the proof we assume that $M>2e(G)$. Let $w$ and $z$ be two
vertices of $H$. As we have already mentioned our construction
implies that there are uniquely determined corresponding vertices
$w'$ and $z'$ in $MG$.\medskip

Obviously there are two possibilities:
\smallskip

(1) $d_{MG}(w',z')\ge\frac12\,d_H(w,z)$. In this case we observe
that the definitions of $\varphi$ and of the norm on
$W\oplus_1\mathbb{R}$ imply that
\[||\varphi(w)-\varphi(z)||\ge||\varphi_0(w')-\varphi_0(z')||\ge
d_{MG}(w',z')/\lip(\varphi_0^{-1})\ge\frac12d_H(w,z)/\lip(\varphi_0^{-1}).\]

(2) $d_{MG}(w',z')<\frac12\,d_H(w,z)$. This inequality implies
that there is a path joining $w$ and $z$ for which the naturally
defined {\it short-paths-portion} is longer than the {\it
long-paths-portion}. The inequality $M>2e(G)$ implies that the
short-paths-portion of this path consists of one path of length
$>\frac12d_H(w,z)$. This implies that the difference between the
second coordinates of $w$ and $z$ in the decomposition
$W\oplus_1\mathbb{R}$ is $>\frac12d_H(w,z)$. Thus
$||\varphi(w)-\varphi(z)|| >\frac12d_H(w,z)$.
\medskip

Since $\lip(\varphi_0^{-1})\ge 1$ (this follows from the
assumption $\lip(\varphi_0)\le 1$), we get $\lip(\varphi^{-1})\le
2\lip(\varphi_0^{-1})$ in each of the cases (1) and (2). The proof
of Theorem \ref{T:Ost11+} is completed.
\medskip

\section{Metric characterization with one
test-space}\label{S:OneSpace}

Bourgain \cite{Bou86} proved that a Banach space is
nonsuperreflexive if and only if it admits uniformly bilipschitz
embeddings of binary trees of all finite depths (see
\cite[pp.~412, 436]{BL00} for the definition and equivalent
characterizations of superreflexivity). Baudier \cite{Bau07}
strengthened the ``only if'' part of this result by proving that
each nonsuperreflexive Banach space admits a bilipschitz embedding
of an infinite binary tree.
\medskip

Our purpose in this section is to find similar
one-test-space-characterizations for classes of Banach spaces
defined in terms of excluded finite-dimensional subspaces. At this
moment we do not know how to do this for an arbitrary sequence of
finite-dimensional subspaces, we found such a characterization
only for increasing sequences of finite-dimensional subspaces.

\begin{theorem}\label{T:OneSpace} Let $\{X_n\}_{n=1}^\infty$ be an increasing
sequence of finite-dimensional Banach spaces with dimensions going
to $\infty$. Then there exists an infinite graph $G$ such that the
following conditions are equivalent:

\begin{itemize}

\item $G$ admits a bilipschitz embedding into a Banach space $X$.

\item The spaces $\{X_n\}$ admit uniformly isomorphic embeddings
into $X$.

\end{itemize}
\end{theorem}

\begin{remark} Theorem \ref{T:OneSpace} is obviously weaker that
Theorem \ref{T:OneDeg3} proved in Section \ref{S:Deg3}. We give an
independent proof of Theorem \ref{T:OneSpace} because it is
substantially simpler.\end{remark}

\begin{proof}[Proof of Theorem \ref{T:OneSpace}]
Our proof is based on the construction of graphs providing
approximate descriptions of convex sets, see Definition
\ref{D:CAlphaBeta}. We use the following immediate consequence of
Theorem \ref{T:discr}:

\begin{lemma}\label{L:ConseqDiscr} Let $X$ be a finite-dimensional Banach space, $Y$
be a Banach space admitting uniformly bilipschitz embeddings of
some $(B_X(n),1)$-graphs, and $C\in[1,\infty)$ be an upper bound
for distortions of these embeddings. Then for each $\ep>0$ there
is a linear embedding $T:X\to Y$ satisfying
$||T||\cdot||T^{-1}||\le 3(1+\ep)C$.
\end{lemma}

Let $L$ be the inductive limit of the sequence $\{X_n\}$, that is,
$L=\bigcup_{n=1}^\infty X_n$ with its natural vector operations
and the norm whose restriction to each of $X_n$ is the norm of
$X_n$. So $L$ is an incomplete normed space (we can, of course,
consider its completion, but for our purposes completeness is not
needed). We construct $G$ as a graph whose vertex set $V$ is a
countable infinite subset of $L\oplus_1 \mathbb{R}$. (We fix the
$\ell_1$-sum because it is convenient to have a precise formula
for the norm, of course in the bilipschitz category all direct
sums are equivalent.) The main features of the construction are:

\begin{itemize}

\item[(1)] The graph $G$ with its shortest path metric is a
locally finite metric space. (Recall that a metric space is called
{\it locally finite} if all balls of finite radius in it have
finite cardinality.)

\item[(2)] We have $V=\bigcup_{n=1}^\infty V_n$, where $V_n$ are
finite and there exist uniformly bilipschitz embeddings
$f_n:V_n\to (X_n\oplus_1\mathbb{R})$, where we assume that the
distance in $V_n$ is inherited from $G$.

\item[(3)] The set $V$ endowed with its shortest path metric $d_G$
contains images of bilipschitz embeddings of some
$(B_{X_n}(m),1,2)$-graphs (with their shortest path metrics),
$m,n\in\mathbb{N}$, with uniformly bounded distortions.

\end{itemize}

First let us explain why such graph $G$ satisfies the conclusion
of Theorem \ref{T:OneSpace}. Suppose that $X$ is such that the
spaces $\{X_n\}$ admit uniformly isomorphic embeddings into $X$.
Then condition (2) implies that $V_n$ admit uniformly bilipschitz
embeddings into $X$. Since $G$ is locally finite, by the main
result of \cite{Ost11c}, we get that $G$ admits a bilipschitz
embedding into $X$.
\medskip

Now suppose that $G$ admits a bilipschitz embedding into $X$. By
(3) we get that $X$ admits uniformly bilipschitz embeddings of
some $(B_{X_n}(m),1)$-graphs, $m,n\in\mathbb{N}$. Applying Lemma
\ref{L:ConseqDiscr}, we get that the spaces $\{X_n\}$ admit
uniformly isomorphic embeddings into $X$.
\medskip

We construct $V$ as an infinite union $\bigcup_{n=1}^\infty V_n$,
where each $V_n$ is a finite subset of
\[C_n:=\conv\left(\bigcup_{k=1}^n (B_{X_k}(k), s_k)\right)\subset L\oplus_1\mathbb{R},\]
where $s_k\in\mathbb{R}$, the pairs $(z,s_k)$ are in the sense of
the decomposition $L\oplus_1\mathbb{R}$, and $(B_{X_k}(k),
s_k)=\{(z,s_k):~ z\in B_{X_k}(k)\}$. Recall that $B_{X_k}(k)$ is
the centered at $0$ ball of $X_k$ of radius $k$.
\medskip

Now we describe our choice of $V_n$ and $s_n$ such that the
conditions (1)--(3) above are satisfied. One of the requirements
is
\begin{equation}\label{E:NonIntNets}
s_{n+1}-s_n>1.
\end{equation}

We let $s_1=0$ and let $V_1$ to be a $1$-separated $1$-dense
subset of $(B_{X_1}(1), s_1)$ (see Definition
\ref{D:MetrNotions}).\medskip

  The choice of $s_2$ is less
restrictive than further choices. We let $s_2=2$ and let $V_2$ be
the extension of $V_1$ to a $1$-separated $1$-dense subset of
\[C_2:=\conv\left(\bigcup_{n=1}^2 (B_{X_n}(n), s_n)\right)\]
satisfying the condition: some part of $V_2$ is a $1$-dense in
$(B_{X_1}(2), s_2)$, and some part of it is a $1$-dense in
$(B_{X_2}(2), s_2)$ (we use $s_2-s_1>1$).\medskip

We use the following notation for a subsets $A$ of
$L\oplus_1\mathbb{R}$:
\begin{equation}\label{E:NotatInter}
A[a,b]:=A\cap(L\oplus_1[a,b]),\end{equation} where $[a,b]$ is an
interval in $\mathbb{R}$. We use similar notation for open and
half-open intervals.\medskip

Now we turn to the choice of $s_3$. We choose $s_3$ to satisfy
\eqref{E:NonIntNets} and to be so large that $V_2$ is
$\left(1+\frac12\right)$-dense in $C_3[s_1,s_2]$. It is easy to
see that sufficiently large $s_3$ satisfy these conditions. Then
we extend $V_2$ to a $1$-separated $\left(1+\frac12\right)$-dense
in $C_3$ subset $V_3$ in such a way that
\begin{itemize}

\item $(V_3\backslash V_2)\subset C_3(s_2,s_3]$

\item Some parts of $V_3$ are $\left(1+\frac12\right)$-dense in
the sets $(B_{X_1}(3), s_3)$, $(B_{X_2}(3), s_3)$, and
$(B_{X_3}(3), s_3)$, respectively (here we use
\eqref{E:NonIntNets}).

\end{itemize}

We continue in the following way. We pick $s_4$ in such a way that

\begin{itemize}

\item $V_2$ is $\left(1+\frac12+\frac14\right)$-dense in
$C_4[s_1,s_2]$.

\item $V_3$ is $\left(1+\frac12+\frac14\right)$-dense in
$C_4[s_1,s_3]$.

\end{itemize}

Now we extend $V_3$ to a $1$-separated
$\left(1+\frac12+\frac14\right)$-dense subset $V_4$ of $C_4$ in
such a way that
\begin{itemize}

\item $(V_4\backslash V_3)\subset C_4(s_3,s_4]$

\item Some parts of $V_4$ are
$\left(1+\frac12+\frac14\right)$-dense subsets in $(B_{X_1}(4),
s_4)$, $(B_{X_2}(4), s_4)$, $(B_{X_3}(4), s_4)$, and $(B_{X_4}(4),
s_4)$, respectively (here we use \eqref{E:NonIntNets}).

\end{itemize}

We continue in an obvious way: In step $n$ we pick $s_n$,
$s_n-s_{n-1}>1$, in such a way that for each $m=2,\dots,n-1$ we
have:\medskip

$V_m$ is a
$\left(1+\frac12+\frac14+\dots+\left(\frac12\right)^{n-2}\right)$-dense
subset of $C_n[s_1,s_{m}]$.
\medskip

We extend $V_{n-1}$ to a $1$-separated
$\left(1+\frac12+\dots+\left(\frac12\right)^{n-2}\right)$-dense
subset $V_n$ of $C_n$ in such a way that
\begin{itemize}

\item $(V_n\backslash V_{n-1})\subset C_n(s_{n-1},s_n]$

\item Some parts of $V_n$ are
$\left(1+\frac12+\dots+\left(\frac12\right)^{n-2}\right)$-dense
subsets in
\[(B_{X_1}(n), s_n), \dots, (B_{X_n}(n), s_n),\] respectively
(here we use \eqref{E:NonIntNets}).

\end{itemize}
\medskip

Let $V=\bigcup_{n=1}^\infty V_n$ and
\[C=\conv\left(\bigcup_{n=1}^\infty (B_{X_n}(n), s_n)\right)\subset L\oplus_1\mathbb{R}.\]

Our construction implies that $V$ is a $1$-separated $2$-dense
subset of $C$. We let $G$ be the corresponding $(C,1,2)$-graph
(see Definition \ref{D:CAlphaBeta}). It remains to verify that $G$
satisfies the conditions (1)--(3) above.
\medskip

Condition (1). The set $V$ is a locally finite subset of
$L\oplus_1\mathbb{R}$ because it is contained in
$L\oplus_1[s_1,\infty)$ and its intersection with each subset of
the form $L\oplus_1 [s_1,s_n]$ is a finite set $V_n$. The graph
$G$ is locally finite because, by Lemma \ref{L:CAlphaBeta}, its
natural embedding into $L\oplus_1\mathbb{R}$ is bilipschitz.
\medskip

Condition (2). We apply Lemma \ref{L:CAlphaBeta} to $V$, the
corresponding $(C,1,2)$-graph, and $L\oplus\mathbb{R}$. We get
that the natural embedding of $V$ with the metric inherited from
$G$ into $L\oplus\mathbb{R}$ is bilipschitz. Hence its
restrictions to $V_n$ are uniformly bilipschitz. The fact that the
restriction of this map to $V_n$ maps $V_n$ into
$X_n\oplus\mathbb{R}$ follows from the definitions.
\medskip

Condition (3). Our construction of set $V$ is such that it
contains subsets which are $1$-separated $2$-dense subsets in
shifted $B_{X_m}(n)$. We apply Lemma \ref{L:CAlphaBeta} twice.
First time to $V$, the corresponding $(C,1,2)$-graph, and
$L\oplus\mathbb{R}$. Second time we apply it to the set
$(B_{X_m}(n),s_n)$, the corresponding $((B_{X_m}(n),s_n),
1,2)$-graph (vertex set of this graph is the intersection of $V$
with $(B_{X_m}(n),s_n)$). We get that embeddings of all of these
graphs into $L\oplus_1\mathbb{R}$ have uniformly bounded
distortions. Therefore the  metrics of these $((B_{X_m}(n),s_n),
1,2)$-graphs are bilipschitz equivalent to the metrics inherited
from $G$. Hence the condition (3) is also satisfied.
\end{proof}

\section{Characterization in terms of an infinite graph with maximum degree
$3$}\label{S:Deg3}

Our next purpose is to show that the test-space for Theorem
\ref{T:OneSpace} can be chosen to have maximum degree $3$:

\begin{theorem}\label{T:OneDeg3} Let $\{X_n\}_{n=1}^\infty$ be an increasing
sequence of finite-dimensional Banach spaces with dimensions going
to $\infty$. Then there exists an infinite graph $H$ with maximum
degree $3$ such that the following conditions are equivalent:

\begin{itemize}

\item $H$ admits a bilipschitz embedding into a Banach space $X$.

\item The spaces $\{X_n\}$ admit uniformly isomorphic embeddings
into $X$.

\end{itemize}
\end{theorem}

\begin{proof} Our proof uses some of the ideas of the proof of Theorem
\ref{T:OneSpace}. For this reason we keep the same notation for
some of the objects, although now they are somewhat different. We
use Definition \ref{D:MetrNotions} and the notation introduced in
formula \eqref{E:NotatInter}.
\medskip

We may assume without loss of generality that $\dim X_n=n$. We
introduce convex sets $C_n$ and $C$ in $L\oplus\mathbb{R}$ of the
form
\[C_n:=\conv\left(\bigcup_{k=1}^n (B_{X_{i(k)}}(4^k), s_k)\right)~\hbox{ and }~C=\bigcup_{n=1}^\infty C_n,\]
where $\{i(k)\}_{k=1}^\infty$ is a sequence of natural numbers
satisfying $i(1)=1$, $i(k)\le i(k+1)\le i(k)+1$, and such that the
equality $i(k+1)=i(k)+1$ holds {\bf rarely} (the exact condition
will be described later), so the dimension of the sets $C_n$
increases slowly.
\medskip

Let $\{s_n\}_{n=1}^\infty$ be a sequence of real numbers, such
that $s_1=0$ and the following two conditions are satisfied:
\medskip

{\bf Gap condition:}
\begin{equation}\label{E:3times2^nGap}
s_n-s_{n-1}>6^{n+1}.
\end{equation}

{\bf Density condition:}
\begin{equation}\label{E:sSUBn} C_{n-1}[a,s_{n-1}]~\hbox{ is }\left(\frac12\right)^{n-2}-\hbox{dense in }~C_n[a,s_{n-1}]
~\hbox{ for every }~a\in[s_1,s_{n-1}].
\end{equation}

\remove{{\bf Narrowness condition:} For each $(x,a)\in C$ there
exists a number $S<\infty$ such that for each $(y,b)\in C$ with
$b\ge S$ we have
\begin{equation}\label{E:Narrow}
||(x,a)-(y,b)||\le 2|a-b|
\end{equation}
}

\begin{remark} It is easy to verify that condition \eqref{E:sSUBn} \remove{and \eqref{E:Narrow}
are} is satisfied for each sufficiently rapidly increasing
sequence $\{s_n\}_{n=1}^\infty$.
\end{remark}

We construct two sequences of finite subsets in
$L\oplus\mathbb{R}$, $\{A_n\}_{n=1}^\infty$ and
$\{B_n\}_{n=2}^\infty$. The desired properties of these sequences
of sets are the following:\medskip

\begin{enumerate}

\item\label{I:A_nINC_n} $A_n$ is  $2^n$-separated $2^n$-dense
subset in $C_n[s_{n-1}+2^n,s_n]$. If $n=1$, this condition is
replaced by: $A_1$ is a $2$-separated $2$-dense set in $C_1$.
\medskip

\item\label{I:A_nINballs} $A_n$ contains $2^n$-separated
$2^n$-dense subsets in $\{(B_{X_{i(k)}}(4^n),s_n)\}_{k=1}^n$.
\medskip

\item\label{I:B_nETC} $B_n$, $n\ge 2$ is a $2^n$-separated subset
of $C_n(s_{n-1},s_{n-1}+2^n)$ such $d(A_n,B_n)\ge 2^n$,
$d(A_{n-1},B_n)\ge 2^n$ and $A_{n-1}\cup B_n\cup A_n$ is a
$2^{n-1}$-separated $2^n$-dense subset in
$C_n[s_{n-2}+2^{n-1},s_n]$. If $n=2$, the last condition is
replaced by: $2$-separated and $2^2$-dense in $C_2[s_1,s_2]=C_2$.
\medskip
\end{enumerate}

We construct such sets in steps. First we construct $A_n$, then
$B_n$ (for $n\ge 2$). We start by letting $A_1$ to be any
$2$-separated $2$-dense subset of $(B_{X_1}(4),s_1)$.\medskip

The construction of $A_n$ $(n\ge 2)$ starts with picking a
$2^n$-separated $2^n$-dense subset of \[(B_{X_1}(4^n),s_n).\] Then
we gradually extend this subset to $2^n$-separated $2^n$-dense
subsets of
\begin{equation}\label{E:BallsN}(B_{X_{i(2)}}(4^n),s_n), \dots
(B_{X_{i(n)}}(4^n),s_n).\end{equation} Observe that our
description of the sequence $\{i(n)\}$ implies that many sets in
the sequence \eqref{E:BallsN} are the same.
\medskip

We complete the construction of $A_n$ extending the obtained set
to a $2^n$-separated $2^n$-dense subset of
$C_n[s_{n-1}+2^n,s_n]$.\medskip

To construct $B_n$ we remove from $C_n(s_{n-1}, s_{n-1}+2^n)$ all
elements which are covered by $2^n$-balls centered in $A_{n-1}\cup
A_n$. If the obtained set $R$ is empty, we let $B_n=\emptyset$.
Otherwise we let $B_n$ be a $2^n$-separated $2^n$-dense subset of
$R$.\medskip

The only condition which has to be verified is the condition that
$A_{n-1}\cup B_n\cup A_n$ is $2^n$-dense in
$C_n[s_{n-2}+2^{n-1},s_n]$ (and its version for $n=2$). Here we
use the condition \eqref{E:sSUBn}. By this condition, since
$A_{n-1}$ is $2^{n-1}$-dense in
$C_{n-1}[s_{n-2}+2^{n-1},s_{n-1}]$, it is
$\left(2^{n-1}+\left(\frac12\right)^{n-2}\right)$-dense in
$C_{n}[s_{n-2}+2^{n-1},s_{n-1}]$. Since
$\left(2^{n-1}+\left(\frac12\right)^{n-2}\right)<2^n$, the
conclusion follows from the construction of $A_n$ and $B_n$.
\medskip

Let \[V=\left(\bigcup_{n=1}^\infty
A_n\right)\cup\left(\bigcup_{n=2}^\infty B_n\right).\]

We create a weighted graph with the vertex set $V$ by joining a
vertex $v\in (A_n\cup B_n)$, $n\ge 2$, to all vertices of
\[\left(\bigcup_{k=1}^n A_{k}\right)\cup\left(\bigcup_{k=2}^n B_{k}\right)\] which are
within distance  $3\cdot 2^n$ to $v$ in the normed space
$L\oplus_1\mathbb{R}$. (Also we join each vertex  $v\in A_1$ to
all vertices of $A_1$ which are within distance $6$ to $v$.) The
inequality \eqref{E:3times2^nGap} implies that in this way
vertices of $A_n\cup B_n$ are joined only to some of the vertices
in $A_{n-1}\cup B_{n-1}$, $A_n\cup B_n$, and $A_{n+1}\cup
B_{n+1}$. The vertex $v$ is joined to some vertices in
$A_{n+1}\cup B_{n+1}$ if $v$ is within distance $3\cdot 2^{n+1}$
in $L\oplus_1\mathbb{R}$ to those vertices.
\medskip

We assign weight (length) $2^n$ to all edges joining $v\in
(A_n\cup B_n)$ with vertices of $(A_{n-1}\cup B_{n-1})\bigcup
(A_n\cup B_n)$. Thus we assign weight (length) $2^{n+1}$ to the
edges joining $v\in (A_n\cup B_n)$ with vertices of $(A_{n+1}\cup
B_{n+1})$. It is clear that we get a well-defined weighted graph.
We denote the obtained weighted graph by $W$ and endow it with its
(weighted) shortest path metric.
\medskip

We estimate the number of edges incident to $v\in (A_n\cup B_n)$
in the following way. All vertices joined to $v$ by edges are in a
ball of radius $3\cdot 2^{n+1}$ centered at $v$. The distance
between two vertices joined to $v$ is at least $2^{n-1}$ because
all such vertices are in the set $A_{n-1}\cup B_{n-1}\cup A_n\cup
B_n\cup A_{n+1}\cup B_{n+1}$, and it is clear from the conditions
on $A_n$ and $B_n$ that the set $A_{n-1}\cup B_{n-1}\cup A_n\cup
B_n\cup A_{n+1}\cup B_{n+1}$ is $2^{n-1}$-separated. All of the
elements of this set are in $X_{i(n+1)}\oplus\mathbb{R}$, and the
dimension of this space is $d(n)=i(n+1)+1$. Therefore
$X_{i(n+1)}\oplus\mathbb{R}$-balls of radiuses $2^{n-2}$ centered
at points joined to $v$ with an edge have disjoint interiors and
are contained in a ball of radius $3\cdot 2^{n+1}+2^{n-2}$
centered at $v$. Comparing the volumes of the union of the
$2^{n-2}$-balls and the $(3\cdot 2^{n+1}+2^{n-2})$-ball containing
them (it is the standard volumetric argument, see e.g. \cite[Lemma
2.6]{MS86}), we get that the number of vertices adjacent to $v$ is
at most
\[\left(\frac{3\cdot 2^{n+1}+2^{n-2}}{2^{n-2}}\right)^{d(n)}=25^{d(n)}.\]
%\newpage

\begin{lemma}\label{L:distortion3} The natural embedding of $W$ into the
normed space $L\oplus_1\mathbb{R}$ has distortion $\le 3$. More
precisely, its Lipschitz constant is $\le 3$, and the Lipschitz
constant of the inverse map is $\le 1$. \end{lemma}

\begin{proof} The statement about the Lipschitz constant of the natural embedding is immediate.
In fact, ends of an edge of length $2^n$ in $W$ correspond to a
vector in $L\oplus_1\mathbb{R}$ whose length is $\le 3\cdot
2^n$.\medskip

The fact that the Lipschitz constant of the inverse map is $\le 1$
can be proved as follows:
\medskip

Let $a$ and $b$ be vertices of $W$. We may assume that $b\in
(A_n\cup B_n)$ and $a\in (A_k\cup B_k)$ for some $k\le n$. We use
double induction. This means the following: First we prove the
result for $n=1$ using induction on $\lfloor ||a-b||/2\rfloor$
(here we use Lemma \ref{L:CAlphaBeta}).
\medskip

Next, we assume that the result holds for $n=m$ and any $k\le m$.
We show that this assumption can be used to prove the result for
$n=m+1$ using the induction on $\lfloor ||a-b||/(2^{m+1})\rfloor$.
\medskip

So let us follow the described program. If we divide all distances
by $2$, the desired inequality for $a,b\in A_1$ is a special case
of Lemma \ref{L:CAlphaBeta} for $\alpha=\beta=1$.
\medskip

\noindent{\bf Assumption:} Now we assume that we have proved the
statement for all pairs $b\in (A_m\cup B_m)$, $a\in (A_k\cup
B_k)$, $k\le m$.\medskip

We show that this assumption can be used to prove the statement
for $b\in (A_{m+1}\cup B_{m+1})$, $a\in (A_k\cup B_k)$, $k\le
m+1$, using induction on $\lfloor ||a-b||/(2^{m+1})\rfloor$.
\medskip

If $\lfloor ||a-b||/(2^{m+1})\rfloor\le 2$, then $a$ and $b$ are
joined by an edge of length $2^{m+1}$. In addition, $||a-b||\ge
2^{m+1}$ (see the conditions \ref{I:A_nINC_n} and \ref{I:B_nETC}
in the description of $A_n$ and $B_n$). So in this case we get the
desired $d_W(a,b)\le ||a-b||$.\medskip

\noindent{\bf Induction Hypothesis:} Suppose that the statement
has been proved for all pairs $b\in (A_{m+1}\cup B_{m+1})$, $a\in
(A_k\cup B_k)$, $k\le m+1$, and
\begin{equation}\label{E:CondIH}\lfloor ||a-b||/(2^{m+1})\rfloor\le D.\end{equation} We show that this
implies the same conclusion for pairs $b\in (A_{m+1}\cup
B_{m+1})$, $a\in (A_k\cup B_k)$, $k\le m+1$, satisfying $\lfloor
||a-b||/(2^{m+1})\rfloor\le D+1$.
\medskip

We need to consider the case where the inequality $(D+2)\cdot
2^{m+1}> ||a-b||\ge (D+1)\cdot 2^{m+1}$ is satisfied, $D\ge 2$,
$D\in\mathbb{N}$. In such a case let $\widetilde b$ be the vector
on the line segment joining $b$ to $a$ at distance $2\cdot
2^{m+1}$ from $b$.
\medskip

It is clear from the construction that $\widetilde b\in C_{m+1}$.
Therefore (see condition \ref{I:B_nETC} in the list of conditions
on $A_n$ and $B_n$) there is a point $\widehat b\in V$ such that
$||\widehat b-\widetilde b||\le 2^{m+1}$.
%{\bf (Possibly it is worthwhile to make the corresponding remarks after items \ref{I:A_nINC_n}-\ref{I:B_nETC}.)}
We have $||\widehat b-b||\le 3\cdot 2^{m+1}$  and
\begin{equation}\label{E:Step}||\widehat b -a||\le
||b-a||-2^{m+1}.\end{equation} Also we have $\widehat b\in A_m\cup
B_m\cup A_{m+1}\cup B_{m+1}$. Therefore $d_W(\widehat b,b)\le
2^{m+1}$.
\medskip

In the case when $\widehat b\in A_{m+1}\cup B_{m+1}$, we get the
desired conclusion using the Induction Hypothesis as follows: The
inequality \eqref{E:Step} implies that the pair $a,\widehat b$
satisfies the inequality \eqref{E:CondIH}. By the Induction
Hypothesis,
\[||a-\widehat b||\ge d_W(a,\widehat b).\]
Therefore
\[||a-b||\ge 2^{m+1}+||a-\widehat b||\ge  2^{m+1}+d_W(a,\widehat
b)\ge d_W(a,b).\]

It remains to consider the case $\widehat b\in A_{m}\cup B_{m}$.
This case is to be divided into two subcases: $k=m+1$ and $k\le
m$. In the latter case we use the assumption that we have proved
the statement for points in
$\left(\bigcup_{i=1}^mA_i\right)\cup\left(\bigcup_{i=2}^mB_i\right)$.
In the former case we use the Induction Hypothesis.
\end{proof}

The graph $W$ contains $(B_{X_k}(2^n),1)$-graphs (in the sense of
Definition \ref{D:CAlphaBeta}) as subgraphs for all
$k,n\in\mathbb{N}$, with all edges having the same weight of
$2^n$. This follows from the choice of $4^n$ as the diameter of
the ball used in the construction of $C_n$ and the fact that $A_n$
contains subsets which are $2^n$-separated and $2^n$-dense in the
ball $(B(X_{k}(4^n)),s_n)$ for $k\le i(n)$ (by combining these
facts with our definitions). We claim that these
$(B(X_{k}(2^n)),1)$-graphs embed in a uniformly bilipschitz way
into $W$. This claim can be proved in the following way: Applying
Lemma \ref{L:CAlphaBeta} we get that the natural embeddings of the
$(B(X_{k}(4^n)),2^n)$-graphs (constructed using the
$2^n$-separated and $2^n$-dense in the ball $(B(X_{k}(4^n)),s_n)$)
into $L\oplus_1\mathbb{R}$ are uniformly bilipschitz. By Lemma
\ref{L:distortion3}, the natural embedding of $W$ into
$L\oplus_1\mathbb{R}$ is also bilipschitz. The conclusion on
uniformity of bilipschitz embeddings of the
$(B(X_{k}(2^n)),1)$-graphs into $W$ follows. Thus, by Lemma
\ref{L:GraphsToSubspaces}, bilipschitz embeddability of the graph
$W$ into a Banach space $X$ implies that $X$ admits uniformly
isomorphic embeddings of the spaces $\{X_i\}_{i=1}^\infty$.
\medskip

Now we construct an unweighted graph $H$ with maximum degree $3$
whose existence is claimed in Theorem \ref{T:OneDeg3}. The graph
$H$ will admit a bilipschitz embedding into
$(L\oplus_1\mathbb{R})\oplus_1 \mathbb{R}$. The graph $H$ will be
a modification of the weighted graph $W$. As in the construction
of Section \ref{S:Houston}, the graph $H$ consists of long paths
and short paths. Short paths correspond to vertices of $W$, long
paths correspond to (weighted) edges of $W$. The length of a short
path corresponding to a vertex in $A_n\cup B_n$ is $2\cdot
25^{d(n+1)}$. These lengths of short paths are chosen because they
provide a sufficient number of colors for the coloring introduced
in the next paragraph.
\medskip

We are going to color edges of $W$. For our purposes we need a
proper edge coloring (that is, edges having a common end should
have different colors). Of course, since the degrees of $W$ are
unbounded, we need infinitely many colors. Our purpose is to bound
from above the number of colors used for edges incident with a
vertex $v\in A_n\cup B_n$ by $2\cdot 25^{d(n+1)}$. To achieve this
goal we order vertices in $V$ according to the magnitude of their
$\mathbb{R}$-coordinate in the decomposition $L\oplus\mathbb{R}$
(starting with those for which the $\mathbb{R}$-coordinate is
$0$), resolving ties in an arbitrary way.
\medskip

We color edges incident with the first vertex arbitrarily (there
are at most  $25^{d(1)}$ of such edges). For each of the further
vertices in our list we need to color all uncolored edges incident
with them. We do this according to the following procedure. Let
$v\in A_n\cup B_n$ be the next vertex in our list. We pick an
uncolored edge incident with $v$, let $u$ be the other end of this
edge. We cannot use for the edge $vu$ the colors which have
already been used for other edges incident to $v$ and to $u$.
There are at most $25^{d(n)}-1$ colors already used for edges
incident with $v$. As for $u$, we know that (see the construction
of $W$) $u\in A_{n-1}\cup B_{n-1}\cup A_n\cup B_n\cup A_{n+1}\cup
B_{n+1}$, therefore the degree of $u$ is $\le 25^{d(n+1)}$.
Therefore among $2\cdot 25^{d(n+1)}$ colors there should be an
available color for the edge $vu$.
\medskip

Now we create a graph as in Section \ref{S:Houston}. The only
difference is that we paste a long path to the level corresponding
to its color (rather than to the level corresponding to the path
itself as we did in Section \ref{S:Houston}). The length of the
long path corresponding to an edge of $W$ of  weight $2^n$ is
$M\cdot 2^n$, where $M$ is a positive integer which we are going
to select now, together with the sequence $\{i(n)\}_{n=1}^\infty$
(which, as we have already mentioned, is a slowly increasing
sequence). The main condition describing our choice of both
objects is
\begin{equation}\label{E:Mand_i(n)}
M\cdot 2^n>4\cdot 25^{d(n+1)}.
\end{equation}
(Recall that $d(n)=i(n+1)+1$.) This condition ensures that the
length of a long path is larger than the sum of the lengths of the
short paths at its ends.
\medskip

We get an unweighted graph, let us denote it $H$. The maximum
degree of $H$ is $3$ because we use a proper edge coloring. The
graph $W$ admits a bilipschitz embedding into $H$: consider the
map which maps each vertex of $W$ to the vertex of level $1$ on
the corresponding short path. The Lipschitz constant of this
embedding is $\le 2\cdot M$ by \eqref{E:Mand_i(n)}. The Lipschitz
constant of the inverse map is $\le M^{-1}$. In fact, if we
consider two vertices in the image of $W$, and join them by a
shortest path in $H$, the path goes through some collection of
short paths (possibly at the end vertices only). The vertices in
$W$ corresponding to these short paths form a path in $W$. Each
time the weight of the edge in $W$ is $M^{-1}\times($the length of
the corresponding path in $H)$. The conclusion about the Lipschitz
constant of the inverse map follows.\medskip

Observe that the graph $H$ is locally finite because its maximum
degree is $3$. It remains to show that the graph $H$ admits a
bilipschitz embedding into any Banach space $X$ containing
uniformly isomorphic $\{X_i\}$. We do this by proving the fact
that the graph $H$ admits a bilipschitz embedding into $(L\oplus_1
\mathbb{R})\oplus_1\mathbb{R}$. By the finite determination result
of \cite{Ost11c} this is enough because it is easy to see that
finite dimensional subspaces of the space $(L\oplus_1
\mathbb{R})\oplus_1\mathbb{R}$ are uniformly isomorphic to
subspaces in $\{X_i\}$.

\begin{lemma}\label{L:HintoL+R+R} The graph $H$ admits a bilipschitz embedding into $(L\oplus_1
\mathbb{R})\oplus_1\mathbb{R}$.
\end{lemma}

\begin{proof} Vertices of $W$ are in $L\oplus_1
\mathbb{R}$. We map the short path corresponding to a vertex $v\in
A_n\cup B_n$ onto those points of the line segment joining
$(Mv,1)$ with $(Mv,2\cdot 25^{d(n+1)})$ whose second coordinate is
an integer. (The number $M$ is introduced by \eqref{E:Mand_i(n)}.
In a pair $(Mv,a)$ the first component is in $L\oplus_1\mathbb{R}$
and the second component is in the second
$\mathbb{R}$-summand.)\medskip

Now we describe our map for long paths of $H$. For each long path
of $H$, or, what is the same, for each edge $uv$ of $W$ we pick a
vector $x_{uv}\in L$ (the space $L$ is identified with the
corresponding summand in $(L\oplus_1
\mathbb{R})\oplus_1\mathbb{R}$). Suppose that the edge $uv$ has
weight $2^n$ and color $i$ in the coloring above. We number
vertices of the long path corresponding to $uv$ as
$u_0,u_1,\dots,u_N$, where $N=M\cdot 2^n$, $u$ corresponds to
$u_0$ and $v$ corresponds to $u_N$. The image of $u_m$
$(m=0,1,\dots,N)$ under the map which we are constructing is given
by \begin{equation}\label{E:MapLongEdge}Tu_m=\begin{cases}
\left(\left(1-\frac mN\right)Mu+\frac mN\,Mv+m
x_{uv}, i\right) & \hbox{ if }m\le\frac N2\\
\left(\left(1-\frac mN\right)Mu+\frac mN\,Mv+\left(N-m\right)
x_{uv}, i\right) & \hbox{ if }m\ge\frac N2.
\end{cases}\end{equation}
So we map vertices of the long path in $H$ corresponding to an
edge $uv$ of $W$ onto a sequence of evenly distributed points in
the union of two line segments joining $(Mu,i)$ and $(Mv,i)$. We
introduce also a map $O$ given by $Ou_m=\left(1-\frac
mN\right)Mu+\frac mN\,Mv$.
\medskip

The map $T$ introduced by \eqref{E:MapLongEdge} is a Lipschitz map
of the vertex set of $H$ into $(L\oplus_1
\mathbb{R})\oplus_1\mathbb{R}$ for an arbitrary uniformly bounded
set of vectors $\{x_{uv}\}$ in $L$. To show this it suffices to
find a bound for the distances between images of ends of an edge
of $H$. For short-path-edges the distances are equal to $1$
because their ends are mapped onto pairs of the form $(Mv,i)$,
$(Mv,i+1)$. For  a long-path-edge, the distance between the ends
is
\[\left\|\frac1N\,Mv-\frac1N\,Mu\pm x_{uv}\right\|\]
This norm can be estimated from above by $3+\sup_{uv}||x_{uv}||$
(we use the estimate for the Lipschitz constant of Lemma
\ref{L:distortion3}).\medskip

Therefore the purpose of a suitable selection of the set
$\{x_{uv}\}$ is to ensure that $T^{-1}$ is Lipschitz. In a similar
situation in Section \ref{S:Houston} we used Auerbach bases. For
this construction we use a somewhat different type of biorthogonal
sequences. We use systems whose existence was shown by
Ovsepian-Pe\l czy\'nski \cite{OP75}. We mean the following result
proved in \cite{OP75} (see also \cite[p.~44]{LT77}):

\begin{theorem}\label{T:OP75} There is an absolute constant $C>0$ such that for each separable Banach
space $Z$, each sequence $\{f_i^*\}\subset Z^*$, and each sequence
$\{f_i\}\subset Z$ there exists a biorthogonal sequence
$\{z_i,z_i^*\}_{i=1}^\infty$ in $Z$ such that $||z_i||\le C$,
$||z_i^*||\le C$, the linear span of $\{z_i\}$ contains the
sequence $\{f_i\}$, and the linear span of $\{z_i^*\}$ contains
the sequence $\{f_i^*\}$.
\end{theorem}

\begin{remark} Pe\l czy\'nski \cite{Pel76}  and Plichko \cite{Pli76} proved that the constant
$C$ can be chosen to be an arbitrary number $>1$.
\end{remark}\medskip

We apply Theorem \ref{T:OP75} to $Z=L$ in the following situation.
We form the sequence $\{f_i^*\}$ in the following way. We denote
by $P_L:L\oplus\mathbb{R}\to L$ the natural projection. For any
two edges $uv$ and $wz$ in $W$ we consider all vectors of the form
\begin{equation}\label{E:P_Ldiff}
P_L(Ou_m-Ow_p)\end{equation} for all admissible values of $m$ and
$p$. The map $O$ was defined after formula \eqref{E:MapLongEdge}
and $Ow_p$ is the image of a vertex $w_p$ of a long path
corresponding to $wz$.
\medskip

For each vector of the form \eqref{E:P_Ldiff} we pick a supporting
functional $f^*\in L^*$, that is, a functional $f^*$ satisfying
\[f^*(P_L(Ou_m-Ow_p))=||P_L(Ou_m-Ow_p)||\]
and $||f^*||=1$. There are countably many such functionals, so we
can form a sequence $\{f^*_i\}$ containing all of them. Also we
form a sequence $\{f_i\}$ containing all of the vectors of the
form \eqref{E:P_Ldiff}.\medskip

\remove{\begin{definition}\label{D:FarAway} Let $uv$ and $wz$ be
two edges in $W$, so the vertices $u,v,w,z$ are vectors in
$L\oplus\mathbb{R}$. We say that the edges $uv$ and $wz$ are {\it
far away} from each other if for any difference of the form
$O(u_m)-O(w_p)$ (as defined above) we have
\[||P_L(O(u_m)-O(w_p))||\le\frac12||O(u_m)-O(w_p)||.
\]
\end{definition}

\begin{remark} Combining the Narrowness Condition \eqref{E:Narrow} and the
definition of $W$ we get that for each edge $uv$ in $W$ there are
only finitely many edges in $W$ which are not far away from $uv$.
\end{remark}
}

Now we describe a suitable choice of the vectors $x_{uv}$ for
\eqref{E:MapLongEdge}. We enumerate edges of $W$ in the
non-decreasing order of the larger $\mathbb{R}$-coordinates of
their ends, resolving ties arbitrarily. Let $uv$ be the first edge
in the ordering. We pick as $x_{uv}$ an element of the sequence
$\{z_i\}$ satisfying the following two conditions:

\begin{itemize}

\item $z_i$ is annihilated by all functionals $f_j^*$ of the
sequence $\{f_i^*\}$ supporting vectors $P_L(Ou_m-Ow_p)$, where
$u_m$ is in the long path corresponding to $uv$ and $w_p$ is in
the long path corresponding to an edge $wz$ for which the smaller
$\mathbb{R}$-coordinate of its ends is $\le s_3$.

\item $x^*_{uv}:=z_i^*$ annihilates all vectors $f_j$ of the
sequence $\{f_i\}$ of the form $P_L(Ou_m-Ow_p)$, where $u_m$ is in
the long path corresponding to $uv$ and $w_p$ is in the long path
corresponding to an edge $wz$ for which the smaller
$\mathbb{R}$-coordinate of its ends is $\le s_3$.

\end{itemize}

Such pair $z_i,z_i^*$ exists because there are finitely many $f_j$
and $f_j^*$ satisfying the conditions. (Here we use the following
conditions of Theorem \ref{T:OP75}: the linear span of $\{z_i\}$
contains the sequence $\{f_i\}$, and the linear span of
$\{z_i^*\}$ contains the sequence $\{f_i^*\}$.)\medskip

We make a similar choice of $x_{uv}$ for all further edges in the
ordering. More details: If we consider an edge for which the
larger $\mathbb{R}$-coordinate of an end is in the interval
$(s_{n-1}, s_n]$, we replace $s_3$ by $s_{n+2}$ in the conditions
above. Also we pick different pairs $z_i,z_i^*$ for different
edges $uv$.\medskip

With this choice of vectors $x_{uv}$, let us estimate from above
the Lipschitz constant of the inverse of $T$. Let $x,y$ be two
vertices of $H$. Let $x$ be on a long path joining $u$ and $v$ and
$y$ be on a long path joining $w$ and $z$ (this is a generic
description because $x$ and $y$ are allowed to be the end vertices
of the long paths). We need to estimate from above the quotient
\begin{equation}\label{E:quotient}\frac{d_H(x,y)}{||Tx-Ty||}.\end{equation}

We consider a shortest $xy$-path in $H$. It has naturally defined
{\it short-path portion} and {\it long-path portion}. There are
two cases: (1) The length of the short-path portion of this path
has length $\ge\frac12d_H(x,y)$; The length of the long-path
portion of of this path has length $>\frac12d_H(x,y)$.
\medskip

The construction of the graph $H$ (see inequality
\eqref{E:Mand_i(n)}) is such that in the case (1) the short-path
portion consists of just one piece. Let the short path portion
start at level (color) $i$ and end at level $j$. Then $|i-j|\ge
\frac12d_H(x,y)$. On the other hand, since the sum $(L\oplus_1
\mathbb{R})\oplus_1\mathbb{R}$ is direct, we have
$||Tx-Ty||\ge|i-j|\ge\frac12d_H(x,y)$.\medskip

In the case (2) we ignore the difference in the second
$\mathbb{R}$-coordinate (caused by the different colors of the
edges $uv$ and $wz$). There are two subcases: Subcase (A): The
vertices $x$ and $y$ are on the same long path; Subcase (B): The
vertices $x$ and $y$ are on different long paths.
\medskip

\noindent{\bf Subcase (A):} Observe that our construction is such
that there is a functional  supporting $P_L(u-v)$ (let us denote
it $f^*_{uv}$) which annihilates by $x_{uv}$ (because $u-v$ is of
the form $Ou_m-Ow_p$). Let $x=u_p$, $y=u_s$, we have
$d_H(x,y)=|s-p|$ and
\[||Tu_p-Tu_s||\ge |f^*_{uv}P_L(Ou_p-Ou_s)|=\frac{M|s-p|}{N}||P_L(u-v)||.
\]

If $||P_L(u-v)||=||P_Lu-P_Lv||$ is a nontrivially large part of
$||u-v||\ge\frac{N}{M}$, we get the desired estimate.\medskip

If $||P_Lu-P_Lv||$ is only a small part of $||u-v||$, then the
difference between the $\mathbb{R}$-coordinates of $u$ and $v$ is
the large part of $||u-v||$. Denoting the projection of
$L\oplus\mathbb{R}$ to $\mathbb{R}$ by $P_\mathbb{R}$, we use
$P_\mathbb{R}(x_{uv})=0$ and get
\[||Tu_p-Tu_s||\ge\frac{M|s-p|}{N}
|P_\mathbb{R}u-P_\mathbb{R}v|,\] so we get the estimate in this
case, too.
\medskip

\noindent{\bf Subcase (B):} Let $x$ be on a long path
corresponding to an edge $uv$ in $W$, and $y$ be on a long path
corresponding to $wz$. Then there are two possibilities:
\medskip

(i) One of the edges $uv$ and $wz$ has the largest
$\mathbb{R}$-coordinate of its ends in the interval $(s_{n-1},
s_n]$, and the other edge has the least $\mathbb{R}$-coordinate of
its ends in the interval $[s_{m},s_{m+1})$, where $m\ge n+2$
\medskip

(ii) It is not the case.
\medskip

\noindent{\bf Subsubcase (i):} Let $x=u_m$ and $y=w_p$. Ignoring
the second $\mathbb{R}$-coordinate (in the sum
$L\oplus\mathbb{R}\oplus\mathbb{R}$) and using
$P_\mathbb{R}x_{uv}=P_\mathbb{R}x_{wz}=0$ we get
\begin{equation}\label{E:below}||Tx-Ty||\ge M|P_\mathbb{R}u-P_\mathbb{R}w|\ge
M(s_{m}-s_n),\end{equation} where $u$ and $w$ are the
corresponding vertices picked in such a way that $u$ has larger
$\mathbb{R}$-coordinate than $v$ and $w$ has smaller
$\mathbb{R}$-coordinate than $w$. On the other hand

\[\begin{split}
d_H(x,y)&\le d_H(u,w)+M2^n+M2^{m+1}\\&\le
M||u-w||+M2^n+M2^{m+1}\\&\le
M|P_\mathbb{R}u-P_\mathbb{R}w|+M2^n+M2^{m+1}+M4^n+M4^{m+1}\\
&\le M|P_\mathbb{R}u-P_\mathbb{R}w|+M(s_{m}-s_n)\\
&\le 2 M|P_\mathbb{R}u-P_\mathbb{R}w|.
\end{split}\]
To get these inequalities we use
\begin{itemize}

\item The triangle inequality in $H$ for the first inequality.

\item Lemma \ref{L:distortion3} for the second inequality.

\item The observation that $||P_Lz||\le 4^n$ if $z\in C_n$ for the
third inequality (see the definition of $C_n$).

\item The gap condition \eqref{E:3times2^nGap} for the fourth
inequality.

\item The second inequality in \eqref{E:below} for the fifth
inequality.

\end{itemize}

The conclusion follows.\medskip

\noindent{\bf Subsubcase (ii):} We may assume without loss of
generality that $x$ is closer (in $H$) to the short path
corresponding to $u$ than to the short path corresponding to $v$
and that $y$ is closer to the short path corresponding to $w$
rather than to the short path corresponding to $z$. We have
\begin{equation}\label{E:FirstStep}\begin{split}||Tx&-Ty||\\
&\ge \left\| \left(1-\frac mN\right)Mu+\frac mN\,Mv+m x_{uv}-
\left(1-\frac {p}{N'}\right)Mw-\frac {p}{N'}\,Mz-p x_{wz}
\right\|,\end{split}
\end{equation}
where $N'$ is the length of the long path corresponding to $wz$.
Let us denote the vector whose norm is taken in the right-hand
side of \eqref{E:FirstStep} by $B$. We get
\[||Tx-Ty||\ge |x_{uv}^*(B)|=m.\]
\[||Tx-Ty||\ge |x_{wz}^*(B)|=p.\]
Writing $x_{uv}^*(B),x_{wz}^*(B)$ we mean that the functionals
$x_{uv}^*,x_{wz}^*\in L^*$ act on a vector $t\in
L\oplus\mathbb{R}$ by acting on $P_Lt$. We use the fact that the
difference $\left(1-\frac mN\right)Mu+\frac mN\,Mv- \left(1-\frac
{p}{N'}\right)Mw-\frac {p}{N'}\,Mz$ is of the from $Ou_m-Ow_p$ for
edges satisfying the conditions above. Thus this difference is
annihilated by $x_{uv}^*$ and $x_{wz}^*$.\medskip

We apply the triangle inequality to \eqref{E:FirstStep} and get
\[\begin{split}||Tx&-Ty||\\
&\ge M||u-w||-\frac{mM}N||u-v||-\frac{pM}{N'}||w-z||-mC-pC\\
&=M||u-w||-m\left(\frac{M}N||u-v||+C\right)-p\left(\frac{M}{N'}||w-z||+C\right)\\
&\ge
d_H(u,w)-m\left(\frac{M}N||u-v||+C\right)-p\left(\frac{M}{N'}||w-z||+C\right),\end{split}
\]
where $C=\sup_i||z_i||=\sup_{u,v}||x_{uv}||$. (We used Lemma
\ref{L:distortion3} to get the last inequality.) Observe that the
numbers in brackets in the last line are bounded by an absolute
constant, let us denote it by $D$. We also have $d_H(x,y)\le
d_H(u,w)+m+p$. Therefore
\[\frac{d_H(x,y)}{||Tx-Ty||}\le\min\left\{\frac{d_H(u,w)+m+p}{d_H(u,w)-D(m+p)},
\frac{d_H(u,w)+m+p}{m}, \frac{d_H(u,w)+m+p}{p}\right\}.\] It is
easy to see that the minimum in the last formula is bounded from
above by an absolute constant.

\end{proof}
\end{proof}

\end{large}

\begin{small}

\renewcommand{\refname}{
\end{small}
\medskip

\noindent{\sc Department of Mathematics and Computer Science\\
St. John's University\\ 8000 Utopia Parkway, Queens, NY 11439,
USA}\\
e-mail: {\tt ostrovsm@stjohns.edu}

\end{document}